\pdfoutput=1

\documentclass[12pt]{amsart}

\numberwithin{equation}{section}
\usepackage{amssymb}
\usepackage{amsmath}
\usepackage{amsfonts}

\usepackage{mathrsfs}

\usepackage[margin=1.5 in]{geometry}

\newtheorem{theorem}{Theorem}[section]
\newtheorem{lemma}[theorem]{Lemma}

\newtheorem{corollary}[theorem]{Corollary}

\theoremstyle{definiton}
\newtheorem*{remarks}{Remarks}
\newtheorem*{note}{Note}
\newtheorem*{remark}{Remark}

\newcommand{\F}{{\mathbb F}}
\newcommand{\I}{{\mathbb I}}

\newcommand{\Z}{{\mathbb Z}}

\newcommand{\h}{{\mathscr H}}
\newcommand{\f}{{\mathcal F}}
\newcommand{\p}{{\mathcal P}}
\newcommand{\g}{{\mathbb H}}
\newcommand{\pp} {{\mathbb P}}

\DeclareMathOperator{\tr}{tr}
\DeclareMathOperator{\sgn}{sgn}
\DeclareMathOperator{\USp}{USp}
\DeclareMathOperator{\Zeta} {Z}
\DeclareMathOperator{\e} {e}
\DeclareMathOperator{\E} {\acute{e}}
\DeclareMathOperator{\A} {\grave{a}}

\title[Traces of High Powers]{Traces of High Powers of the Frobenius Class in the Moduli Space of Hyperelliptic Curves}

\author{Iakovos Jake Chinis}

\date{\today}

\begin{document}

\begin{abstract}
The Zeta function of a curve $C$ over a finite field may be expressed in terms of the characteristic polynomial of a unitary matrix $\Theta_C$. Following the work of Rudnick \cite{Rud10}, we compute the expected value of $\tr(\Theta_C^n)$ over the moduli space of hyperelliptic curves of genus $g$, over a fixed finite field $\F_q$, in the limit of large genus. As an application, we compute the expected value of the number of points on $C$ in $\F_{q^n}$ as the genus tends to infinity. We also look at biases in both expected values for small values of $n$.
\end{abstract}

\maketitle

\section{Introduction}

Let $C$ be a smooth projective curve of genus $g \geq 1$ defined over a fixed finite field $\F_q$ of odd cardinality $q$. If we let $\#C(\F_{q^n})$ denote the number of points on $C$ in finite extensions $\F_{q^n}$ of degree $n$ of $\F_q$, then the Zeta function associated to the curve $C$ is defined by
\begin{align}
\label{Zeta_1}
\Zeta_C(u):=\exp \Big( \sum_{n=1}^{\infty} \frac{\#C(\F_{q^n})}{n} u^n \Big), |u| < \frac{1}{q}.
\end{align}
It is known that $\Zeta_C(u)$ is a rational function in $u$ of the form
\begin{align*}
\Zeta_C(u)=\frac{P_C(u)}{(1-u)(1-qu)},
\end{align*}
where $P_C(u) \in \Z[u]$ is a polynomial of degree $2g$, with $P_C(0)=1$, satisfying the functional equation
\begin{align}
\label{FunctionalEquation_1}
P_C(u)=(qu^2)^gP_C(\frac{1}{qu}).
\end{align}
It was proven by Weil \cite{Weil} that the zeros of $P_C(u)$ all lie on the circle $|u|=1/q^\frac{1}{2}$. Hence,
\begin{align*}
P_C(u)=\prod_{j=1}^{2g}(1-q^\frac{1}{2} \e^{i\theta_j(C)}u),
\end{align*}
for some angles $\theta_j(C)$, $1\leq j \leq 2g$, and
\begin{align}
\label{Zeta_2}
\Zeta_C(u) = \exp \Big( \sum_{n=1}^{\infty} \frac{\#C(\F_{q^n})}{n} u^n \Big)
		   = \frac{\prod_{j=1}^{2g}(1-q^\frac{1}{2} \e^{i\theta_j(C)}u)}{(1-u)(1-qu)}.
\end{align}  
		   
Now, we may define a unitary symplectic matrix $\Theta_C \in \USp(2g)$ by
\begin{align*}
\Theta_{C_{jk}}:= \left\{\def\arraystretch{1.2}%
\begin{array}{@{}c@{\quad}l@{}}
\e^{i\theta_j(C)} & \mbox{if $k=j$}\\
0 & \mbox{otherwise,}\\
\end{array}\right.
\end{align*}
for $1 \leq j,k \leq 2g$. Then it is clear that the Zeta function associated to $C$ can be expressed in terms of the characteristic polynomial of $\Theta_C$
\begin{align*}
\Zeta_C(u)=\frac{\det(\I-u\sqrt{q}\Theta_C)}{(1-u)(1-qu)},
\end{align*}
with $\Theta_C$ unique up to conjugacy. We call the conjugacy class of $\Theta_C$ the \textit{unitarized Frobenius class of $C$}.

Let $\h_g$ be the moduli space of hyperelliptic curves of genus $g$ over $\F_q$; i.e., the set of hyperelliptic curves given by the affine equation
\begin{align*}
C_Q: y^2=Q(x),
\end{align*}
where $Q\in \F_q[x]$ is any squarefree polynomial of degree $2g+1$ or $2g+2$. In this model, the point at infinity is not smooth, but we may consider the smooth portion of $C_Q$ and account for the point at infinity separately; in the smooth model, the point at infinity will be replaced by 0, 1, or 2 points and we get that the number of points at infinity is
\begin{align}
\label{infinity}
\left\{\def\arraystretch{1.2}%
\begin{array}{@{}r@{\quad}l@{}}
0 & \mbox{if $\deg(Q)$ is even and $\sgn(Q)\neq \square$}\\
1 & \mbox{if $\deg(Q)$ is odd}\\
2 & \mbox{if $\deg(Q)$ is even and $\sgn(Q)=\square$.}\\
\end{array}\right.
\end{align}
Note the relation between equations (\ref{infinity}) and (\ref{lambda}); namely, the number of points at infinity is $\lambda_Q+1$, with $\lambda_Q$ as in (\ref{lambda}).

\begin{remark}
The smooth model is the closure of $C_Q$, denoted $\overline{C_Q}$, under the map
\begin{align*}
[1,x,x^2,\dots,x^{g-1},y]:C_Q\rightarrow \pp^{g+2}.
\end{align*}
One can show that this closure consists of two affine components: the first is $C_Q$ itself and the second is the curve given by $y^2=x^{2g+2}Q(\frac{1}{x})$. In fact, $C_Q$ is isomorphic to $\overline{C_Q}\bigcap\{x_0\neq 0\}$; we refer the reader to Silverman \cite{Silverman}. 
\end{remark}

For any function $F$ on $\h_g$, we define the expected value of $F$ over $\h_g$ 
\begin{align}
\label{first}
\langle F \rangle_{\h_g}:=\frac{1}{\#\h_g} \cdot \sum_{C_Q \in \h_g} F(C_Q).
\end{align}

In this paper, we study the traces of high powers of the Frobenius class of $C_Q$ over $\h_g$ over a fixed finite field $\F_q$ of odd cardinality $q$ as  $g \rightarrow \infty$. In particular, we concern ourselves with the expected values of $\langle \tr(\Theta_{C_Q}^n) \rangle_{\h_g}$ as $g \rightarrow \infty$ and we compare our work with the Random Matrix results of \cite{RMT}.

From the work of Diaconis and Shahshahani \cite{RMT}, the expected value of the traces of powers over the unitarized symplectic group $\USp(2g)$ is given by
\begin{align}
\label{RMT}
\int_{\USp(2g)}\tr(U^n)du= \left\{\def\arraystretch{1.2}%
\begin{array}{@{}c@{\quad}l@{}}
-\eta_n & \mbox{if $1\leq n \leq 2g$}\\
0 & \mbox{if $n > 2g$,}\\
\end{array}\right.
\end{align}
where 
\[
\eta_n= \left\{\def\arraystretch{1.2}%
\begin{array}{@{}c@{\quad}l@{}}
1 & \mbox{if $n$ is even}\\
0 & \mbox{if $n$ is odd.}\\
\end{array}\right.
\]
We will prove the following theorem and the accompanying corollary:
\begin{theorem}
\label{THM_1}
For $n$ odd,
\begin{align*}
\langle \tr(\Theta_C^n) \rangle_{\h_g}=0,
\end{align*}
and for $n$ even,
\begin{align*}
\langle \tr(\Theta_C^n) \rangle_{\h_g}= 
\frac{1}{q^{\frac{n}{2}}}\cdot\sum_{\substack{{\deg(P)|\frac{n}{2}}\\{\deg(P)\neq 1}}} \frac{\deg(P)}{|P|+1} +O(gq^{\frac{-g}{2}})+
\left\{\def\arraystretch{1.2}%
\begin{array}{@{}l@{\quad}l@{}}
-1 & \mbox{$0 < n < 2g$}\\
-1-\frac{1}{q^2-1} & \mbox{$n = 2g$}\\
O(nq^{\frac{n}{2}-2g}) & \mbox{$2g < n$,}\\
\end{array}\right.
\end{align*}
where the sum is over all monic irreducible polynomials $P \in \F_q[x]$ and where $|P|:=q^{\deg(P)}$.
\end{theorem}

\begin{corollary}
If $n$ is odd, then
\begin{align*}
\langle \tr(\Theta_C^n) \rangle_{\h_g} = \int_{\USp(2g)} \tr(U^n) dU.
\end{align*}
For $n$ even with $3 \log_q(g) < n < 4g-5 \log_q(g)$ and $n \neq 2g$,
\begin{align*}
\langle \tr(\Theta_C^n) \rangle_{\h_g} = \int_{\USp(2g)} \tr(U^n) dU + o(\frac{1}{g}).
\end{align*} 
\end{corollary}

In \cite{Rud10}, Rudnick considers the mean value of $\tr(\Theta_C^n)$ over a family of hyperlliptic curves given by the affine equation $C: y^2=Q(x)$ where $Q(x)\in \f_{2g+1}$, with
\begin{align*}
\f_{2g+1}:=\{f\in \F_q[x] : \mbox{ $f$ monic, squarefree, and $\deg(f)=2g+1$}\},
\end{align*}
and obtains that
\begin{align*}
\langle \tr (\Theta_C^n) \rangle_{\f_{2g+1}} = 
\eta_n\frac{1}{q^{\frac{n}{2}}}\sum_{\substack{{\deg(P)|\frac{n}{2}}}} \frac{\deg(P)}{|P|+1} 
+O(gq^{-g})+
\left\{\def\arraystretch{1.2}%
\begin{array}{@{}l@{\quad}l@{}}
-\eta_n & \mbox{$0 < n < 2g$}\\
-1-\frac{1}{q-1} & \mbox{$n = 2g$}\\
O(nq^{\frac{n}{2}-2g}) & \mbox{$2g < n$;}\\
\end{array}\right.
\end{align*}
in particular, if $3 \log_q(g) < n < 4g-5 \log_q(g)$ and $n \neq 2g$, then
\begin{align*}
\langle \tr (\Theta_C^n) \rangle_{\f_{2g+1}}=\int_{\USp(2g)} \tr(U^n) dU + o(\frac{1}{g}).
\end{align*}
Rudnick then points out that there is a slight deviation in $\langle \tr (\Theta_C^n) \rangle_{\f_{2g+1}}$ from the Random Matrix Theory results for small values of $n$ and for $n=2g$; namely,
\begin{align*}
\langle \tr (\Theta_C^2) \rangle_{\f_{2g+1}} \sim \int_{\USp(2g)} \tr(U^2) dU + \frac{1}{q+1}
\end{align*}
and
\begin{align*}
\langle \tr (\Theta_C^{2g}) \rangle_{\f_{2g+1}} \sim \int_{\USp(2g)} \tr(U^{2g}) dU - \frac{1}{q-1}.
\end{align*}
By considering the average value of $\tr(\Theta_C^n)$ over $\h_g$, we no longer get a deviation from the RMT results for $n=2$ and the deviation at $n=2g$ diminishes:
\begin{align*}
\langle \tr (\Theta_C^{2}) \rangle_{\h_g} \sim \int_{\USp(2g)} \tr(U^{2}) dU 
\end{align*} 
and
\begin{align*}
\langle \tr (\Theta_C^{2g}) \rangle_{\h_g} \sim \int_{\USp(2g)} \tr(U^{2g}) dU - \frac{1}{q^2-1}.
\end{align*}
Furthermore, our results for odd $n$ are exact and coincide with the RMT results for all values of $g$. At first glance, this may seem counterintuitive as one expects to have an error term, as in the even case and as in \cite{Rud10}. Using another approach, one can quickly verify the first result of Theorem \ref{THM_1} (this is done in section \ref{section10}).\\

Now, we may apply Theorem \ref{THM_1} to compute the average number of points on $C_Q$ in finite extensions $\F_{q^n}$ of $\F_q$ over $\h_g$, denoted $\langle \#C(\F_{q^n}) \rangle_{\h_g}$:

By taking logarithmic derivatives in (\ref{Zeta_2}),
\begin{align*}
\#C_Q(\F_{q^n})&=q^n+1-q^\frac{n}{2}\sum_{j=1}^{2g}\e^{in\theta_j(C_Q)}\\
&=q^n+1-q^\frac{n}{2}\tr(\Theta_{C_Q}^n).
\end{align*}
In fact,
\begin{align*}
\langle \#C_Q(\F_{q^n}) \rangle_{\h_g}&=q^n+1-q^\frac{n}{2} \langle \tr(\Theta_{C_Q}^n) \rangle_{\h_g}\\
&\sim q^n+\eta_nq^\frac{n}{2}+1-\eta_n\sum_{\substack{{\deg(P)|\frac{n}{2}}\\{\deg(P)\neq 1}}}\frac{\deg(P)}{|P|+1}.
\end{align*}
More precisely,
\begin{corollary}
(i)  If $n$ is odd, then 
\begin{align*}
\langle \#C_Q(\F_{q^n}) \rangle_{\h_g}=q^n+1.
\end{align*}
(ii) If $n$ is even, then 
\begin{align*}
\langle \#C_Q(\F_{q^n}) \rangle_{\h_g} \sim q^n+q^\frac{n}{2}+1-\sum_{\substack{{\deg(P)|\frac{n}{2}}\\{\deg(P)\neq 1}}}\frac{\deg(P)}{|P|+1}.
\end{align*}
\end{corollary}

Once again, our results for odd $n$ are exact and hold for all values of $g$. Although we continue to get deviations from the RMT results for even $n\geq 4$, our results hold for $n=2$ and our deviations are different from those obtained in \cite{Rud10}.\\

Another approach to computing $\langle \#C_Q(\F_{q^n}) \rangle_{\h_g}$ is the work of Alzahrani \cite{Alz15} who uses the distribution of points on $\h_g$ over $\F_q$ in $\F_{q^n}$. Using these methods, the results of Alzahrani agree with the Corollary above (albeit with a larger error term).\\

Finally, we would like to mention that some of the computations done in sections \ref{section3} through \ref{section8} were done independently by E. Lorenzo, G. Meleleo, and P. Milione in their study of statistics for biquadratic curves; their work is collected in \cite{thesis}.

\section{Background}

In this section, we establish some notation and we introduce the main results of \cite{Rud10}. Since the majority of what follows is based off of the work in \cite{Rud10}, we use the same notation and list important results for the convenience of the reader. We use \cite{Ros02} as a general reference.

Throughout this paper, $\F_q$ is a fixed finite field of odd cardinality $q$, $P$ represents monic irreducible polynomials in $\F_q[x]$, and $Q$ will be used to denote squarefree polynomials of degree $2g+1$ or $2g+2$ with $g\geq 1$. Unless otherwise stated, it is understood that sums and products are over all monic elements in $\F_q[x]$; in the case where a sum involves elements $B\in\F_q[x]$ that are not necessarily monic, we write the sum over $\mbox{$B$ n.n.m.}$.

Given any polynomial $D\in \F_q[x]$ that is not a perfect square, we define the quadratic character $\chi_D$ by the quadratic residue symbol for $\F_q[x]$
\begin{align*}
\chi_D(f):=\Big( \frac{D}{f} \Big),
\end{align*}
where $f$ is any monic polynomial in $\F_q[x]$. 

The Zeta function associated to the hyperelliptic curve $C_Q:y^2=Q(x)$ is then given by
\begin{align*}
\Zeta_{C_Q}(u)=L^*(u,\chi_Q)\zeta_q(u),
\end{align*}
where
\begin{align*}
\zeta_q(u):=\frac{1}{(1-u)(1-qu)}
\end{align*}
is the Zeta function of $\F_q(x)$ and where
\begin{align}
\label{L_1}
L^*(u,\chi_Q)&:=(1-\lambda_Q\cdot u)^{-1}\prod_{P\in \F_q[x]}(1-\chi_Q(P)\cdot u^{\deg(P)})^{-1}\\
\label{L_2}
&=\det(\I-u\sqrt{q}\cdot \Theta_{C_Q}),
\end{align}
with
\begin{align}
\label{lambda}
\lambda_Q:=\left\{\def\arraystretch{1.2}%
\begin{array}{@{}r@{\quad}l@{}}
-1 & \mbox{if $\deg(Q)$ is even and $\sgn(Q)\neq \square$}\\
0 & \mbox{if $\deg(Q)$ is odd}\\
1 & \mbox{if $\deg(Q)$ is even and $\sgn(Q)=\square$,}\\
\end{array}\right.
\end{align}
which relates to the count in equation (\ref{infinity}).

Taking logarithmic derivatives in equations (\ref{L_1}) and (\ref{L_2}), we see that
\begin{align}
\label{trace}
\sum_{j=1}^{2g} \e^{in\theta_j(C_Q)} = \tr(\Theta_{C_Q}^n)=-\frac{\lambda_Q^n}{q^\frac{n}{2}}-\frac{1}{q^\frac{n}{2}}\sum_{\deg(f)=n} \Lambda(f) \chi_Q(f),
\end{align}
where
\begin{align*}
\Lambda(f):=\left\{\def\arraystretch{1.2}%
\begin{array}{@{}c@{\quad}l@{}}
\deg(P) & \mbox{if $f=P^k$}\\
0 & \mbox{otherwise}\\
\end{array}\right.
\end{align*}
is the \textit{von Mangoldt} function.  
  
Let
\begin{align*}
\f_d:=\{f\in\F_q[x]:\mbox{$f$ monic, squarefree, and $\deg(f)=d$}\}
\end{align*}
and let
\begin{align*}
\widehat{\f}_d:=\{f\in\F_q[x]:\mbox{$f$ squarefree and $\deg(f)=d$}\}.
\end{align*}
Then 
\begin{align*}
\#\widehat{\f}_d=(q-1)\#\f_d
\end{align*}
and it is easy to see that (see Lemma 3 of \cite{KurRud}, for example)
\begin{align*}
\#\f_d=\left\{\def\arraystretch{1.2}%
\begin{array}{@{}c@{\quad}l@{}}
(1-\frac{1}{q})q^d, & d\geq 2\\
q, & d=1.\\
\end{array}\right.
\end{align*}

Using these sets of polynomials, every curve in the moduli space of hyperelliptic curves of genus $g$ has a model
$C_Q:y^2=Q(x)$, where $Q\in\widehat{\f}_{2g+1}\bigcup \widehat{\f}_{2g+2}$.

Now, let $\f$ be any family of squarefree polynomials of degree $d$ in $\F_q[x]$. For any function $F$ on $\f$, we define the expected value of $F$ over $\f$
\begin{align*}
\langle F \rangle_\f:=\frac{1}{\#\f}\cdot \sum_{Q\in \f}F(Q).
\end{align*}
In particular,
\begin{align}
\label{ave_gen}
\langle \tr(\Theta_{C_Q}^n) \rangle _{\f}=\frac{1}{\#\f}\cdot \sum_{Q\in \f}\Big( -\frac{\lambda_Q^n}{q^\frac{n}{2}}-\frac{1}{q^\frac{n}{2}}\sum_{\deg(f)=n}\Lambda(f) \chi_Q(f)\Big).
\end{align}
In \cite{Rud10}, Rudnick averages the trace over $\f=\f_{2g+1}$. We begin by considering the average over $\f=\f_{2g+2}$ and then obtain the average over $\h_g$ by combining our results and also considering the contribution of the point at infinity which differs on each component $\widehat{\f}_{2g+1}$, $\widehat{\f}_{2g+2}$. 

Let $\mu$ denote the M\"{o}bius function. Since 
\begin{align*}
\sum_{A^2|Q}\mu(A)=\left\{\def\arraystretch{1.2}%
\begin{array}{@{}c@{\quad}l@{}}
1 & \mbox{if $Q$ is squarefree}\\
0 & \mbox{otherwise,}\\
\end{array}\right.
\end{align*}
we may compute the expected value of $F$ by summing over all elements of degree $d$ in $\F_q[x]$ and sieving out the squarefree terms; namely,
\begin{align}
\label{ave}
\langle   F(Q)\rangle_\f = \frac{1}{\#\f}\sum_{2\alpha+\beta=d}\sum_{\substack{{\deg(B)=\beta}\\\text{$B$ n.n.m.}}}\sum_{\deg(A)=\alpha} \mu(A) F(A^2B).
\end{align}

For all $A,$ $B \in \F_q[x]$,
\begin{align*}
\chi_{A^2B}(f)=\Big(\frac{B}{f}\Big)\cdot \Big(\frac{A}{f}\Big)^2=\left\{\def\arraystretch{1.2}%
\begin{array}{@{}c@{\quad}l@{}}
\Big(\frac{B}{f}\Big) & \mbox{if $(A,f)=1$}\\
0 & \mbox{otherwise.}\\
\end{array}\right.
\end{align*}
With that said, taking $F(Q)=\chi_Q$ in equation (\ref{ave}),
\begin{align*}
\langle \chi_Q(f) \rangle_\f=\frac{1}{\#\f}\cdot \sum_{\substack{{2\alpha+\beta=d}\\{\alpha,\beta\geq 0}}} \sigma(f;\alpha)\sum_{\substack{{\deg(B)=\beta}\\{\text{$B$ n.n.m.}}}}\Big(\frac{B}{f}\Big),
\end{align*}
where
\begin{align*}
\sigma(f;\alpha):=\sum_{\substack{{\deg(A)=\alpha}\\{(A,f)=1}}}\mu(A).\\
\end{align*}

We are now is a position to provide the necessary results from \cite{Rud10}:\\

For any $P\in\F_q[x]$ with $\deg(P)=n$, we define
\begin{align*}
\sigma_n(\alpha):=\sigma(P^k;\alpha)=\sum_{\substack{{\deg(A)=\alpha}\\{(A,P^k)=1}}}\mu(A)=\sum_{\substack{{\deg(A)=\alpha}\\{(A,P)=1}}}\mu(A)=\sigma(P;\alpha).
\end{align*}

\begin{lemma} \cite[Lemma 4]{Rud10}
\label{lem_alpha}
(i) For $n=1$,
\begin{align}
\label{lem_alpha_1}
\mbox{$\sigma_1(0)=1,$ $\sigma_1(\alpha)=1-q$ $\:\forall\alpha\geq 1.$}
\end{align}
(ii) If $n\geq 2$, then
\begin{align}
\label{lem_alpha_2}
\sigma_n(\alpha)=\left\{\def\arraystretch{1.2}%
\begin{array}{@{}c@{\quad}l@{}}
1 & \alpha\equiv 0\mod{n}\\
-q & \alpha\equiv 1\mod{n}\\
0 & \mbox{otherwise.}\\
\end{array}\right.
\end{align}
\end{lemma}

Recall that the Dirichlet L-series associated to $\chi_Q$, denoted $L(u,\chi_Q)$ for $|u|<1/q$, is a polynomial in $u$ of degree at most $\deg(Q)-1$ (see Proposition 4.3 of \cite{Ros02}, for example). In fact, 
\begin{align*}
L(u,\chi_Q):=\prod_{P\in \F_q[x]}(1-\chi_Q(P)\cdot u^{\deg(P)})^{-1}=\sum_{\beta\geq 0}A_Q(\beta)u^\beta,
\end{align*}
where 
\begin{align*}
A_Q(\beta):=\sum_{\deg(B)=\beta}\chi_Q(B)
\end{align*}
and $A_Q(\beta)=0$ for $\beta \geq \deg(Q)$.

Let 
\begin{align*}
S(\beta; n):=\sum_{\deg(P)=n}\sum_{\deg(B)=\beta}\Biggr(\frac{B}{P}\Biggr).
\end{align*}
By the \textit{Law of Quadratic Reciprocity} \cite{Ros02},
\begin{align*}
S(\beta; n)=(-1)^{\frac{q-1}{2}\beta n} \sum_{\deg(P)=n}A_P(\beta)
\end{align*}
\begin{align}
\label{S(beta;n)=0}
\Rightarrow \mbox{$S(\beta;n)=0$ $\forall n \leq \beta.$}
\end{align}

We let $\pi_q(n)$ denote the number of monic irreducible polynomials of degree $n$ in $\F_q[x]$. From the\textit{ Prime Polynomial Theorem} \cite{Ros02},
\begin{align*}
\pi_q(n)&:=\#\{P\in \F_q[x]:\deg(P)=n\}\\
&=\frac{q^n}{n}+O\Big(\frac{q^{\frac{n}{2}}}{n}\Big).
\end{align*}

\begin{lemma} \cite[Proposition 7]{Rud10}
\label{S_n}
(i) $n$ odd, $0\leq \beta \leq n-1$:
\begin{align}
\label{S_n_odd}
S(\beta;n)=q^{\beta-\frac{n-1}{2}}S(n-1-\beta;n)
\end{align}
and 
\begin{align}
\label{S(n-1;n)_odd}
S(n-1;n)=\pi_q(n)q^{\frac{n-1}{2}}.
\end{align}
(ii) $n$ even, $1\leq \beta \leq n-2$:
\begin{align}
\label{S_n_even}
S(\beta;n)=q^{\beta-\frac{n}{2}}\Big(-S(n-1-\beta;n)+(q-1)\sum_{j=0}^{n-\beta-2}S(j;n)\Big)
\end{align}
and
\begin{align}
\label{S(n-1;n)_even}
S(n-1;n)=-\pi_q(n)q^{\frac{n-2}{2}}.
\end{align}
\end{lemma}

\begin{lemma} \cite[Lemma 8]{Rud10}
\label{lem_S_beta}
If $\beta<n$, then
\begin{align}
\label{S_beta}
S(\beta;n)=\eta_\beta\pi_q(n)q^{\frac{\beta}{2}}+O(\frac{\beta}{n}q^{\frac{n}{2}+\beta}),
\end{align}
where $\eta_\beta=1$ for $\beta$ even and $\eta_\beta=0$ for $\beta$ odd.
\end{lemma}

\section{Improved Estimate for $S(\beta;n)$ when $\beta$ is even}
\label{section3}

Initially, we concern ourselves with $\langle \tr(\Theta_{C_Q}^n) \rangle_{\f_{2g+2}}$; in doing so, we need to estimate $S(\beta;n)$ for when $\beta$ is even (see sections \ref{section5} and \ref{section7}). The following theorem makes use of Lemmas \ref{S_n} and \ref{lem_S_beta}; it is the analogous result to Proposition 9 of \cite{Rud10} (since Rudnick considers the average value over $\f_{2g+1}$, estimates for $S(\beta;n)$ in \cite{Rud10} involve $\beta$ odd). Furthermore, this result will allow us to compute $\langle \tr(\Theta_{C_Q}^n) \rangle_{\f_{2g+2}}$ for $n$ near $4g$ (just as Proposition 9 in \cite{Rud10} allows Rudnick to compute $\langle \tr(\Theta_{C_Q}^n) \rangle_{\f_{2g+1}}$ for $n$ near $ 4g$). 
   
\begin{theorem}
\label{thm_beta_even}
If $\beta$ is even, $\beta\neq 0$, and $\beta < n$, then
\begin{align}
\label{beta_even}
S(\beta;n)=\pi_q(n)(q^{\frac{\beta}{2}}-\eta_nq^{\beta-\frac{n}{2}})+O(q^n),
\end{align}
where
\begin{align*}
\eta_n=\left\{\def\arraystretch{1.2}%
\begin{array}{@{}c@{\quad}l@{}}
1 & \mbox{$n$ even}\\
0 & \mbox{$n$ odd.}\\
\end{array}\right.
\end{align*}
\end{theorem}

\begin{remarks} (i) The result above is essentially the same as Proposition 9 in \cite{Rud10} with one additonal term; namely, $\pi_q(n)q^\frac{\beta}{2}$.

(ii) As Rudnick points out in \cite{Rud10}, the main tool in proving Theorem \ref{thm_beta_even} is duality; it allows us to improve the error term in estimates of $S(\beta;n)$ and to get results holding for $n<4g$ and not only for $n<2g$. We would like to mention that the duality present in our character sums $S(\beta;n)$ is based on the functional equation (\ref{FunctionalEquation_1})
\begin{align*}
L^*(u,\chi_P)=(uq^2)^{\lfloor{\frac{\deg(P)-1}{2}}\rfloor} L^*(\frac{1}{qu},\chi_P),
\end{align*}
for prime characters $\chi_P$ (see the proof of Proposition 7 in \cite{Rud10}).
\end{remarks}

\begin{proof}
(i) If $n$ is odd, we apply (\ref{S_n_odd}) to $S(\beta;n)$ and then apply (\ref{S_beta}) to $S(n-1-\beta;n)$:
\begin{align*}
S(\beta;n)&=q^{\beta-\frac{n-1}{2}}S(n-1-\beta;n)\\
&=q^{\beta-\frac{n-1}{2}}\Biggr(\pi_q(n)q^\frac{n-1-\beta}{2}+O\Big(\frac{n-1-\beta}{n}q^{\frac{n}{2}+n-1-\beta}\Big)\Biggr)\\
&=\pi_q(n)q^{\frac{\beta}{2}}+O(q^n).
\end{align*}
(ii) If $n$ is even, we apply (\ref{S_n_even}) to $S(\beta;n)$ and then apply (\ref{S_beta}) to $S(n-1-\beta;n)$:
\begin{align*}
S(\beta;n)&=q^{\beta-\frac{n}{2}}\Biggr(-S(n-1-\beta;n)+(q-1)\sum_{j=0}^{n-\beta-2}S(j;n)\Biggr)\\
&=q^{\beta-\frac{n}{2}}\Biggr(O\Big(\frac{n-1-\beta}{n}q^{\frac{n}{2}+n-1-\beta}\Big)+(q-1)\sum_{j=0}^{n-\beta-2}\Big(\eta_j\pi_q(n)q^{\frac{j}{2}}+O(\frac{j}{n}q^{\frac{n}{2}+j})\Big)\Biggr).
\end{align*}
The two error terms are $O(q^n)$. Since both $n$ and $\beta$ are even, $n-\beta-2$ is even and we may rewrite the main term as 
\begin{align*}
\pi_q(n)q^{\beta-\frac{n}{2}}(q-1)\sum_{j=0}^{\frac{n-\beta-2}{2}}q^j.
\end{align*}
Hence,
\begin{align*}
S(\beta;n)&=\pi_q(n)q^{\beta-\frac{n}{2}}(q-1)\sum_{j=0}^{\frac{n-\beta-2}{2}}q^j+O(q^n)\\
&=\pi_q(n)q^{\beta-\frac{n}{2}}(q^{\frac{n-\beta}{2}}-1)+O(q^n)\\
&=\pi_q(n)(q^{\frac{\beta}{2}}-q^{\beta-\frac{n}{2}})+O(q^n).
\end{align*}
\end{proof}

\section{Computing $\tr(\Theta_{C_Q}^n)$ for $Q\in \f_{2g+2}$}
For the time being, we restrict ourselves to $\f_{2g+2}$. Let $Q\in \f_{2g+2}$ and consider the curve $C_Q:y^2=Q(x)$. The trace of the powers of $\Theta_{C_Q}$ is given by equation (\ref{trace}):
\begin{align}
\label{trace_2g+2}
\tr(\Theta_{C_Q}^n)&=-\frac{1}{q^{\frac{n}{2}}}-\frac{1}{q^{\frac{n}{2}}}\sum_{\deg(f)=n}\Lambda(f)\chi_Q(f)\\
&=-\frac{1}{q^{\frac{n}{2}}}-\frac{1}{q^{\frac{n}{2}}}\sum_{\substack{{P,k}\\{\deg(P^k)}=n}}\deg(P)\chi_Q(P^k)\\
&=-\frac{1}{q^{\frac{n}{2}}}+\p_n+\square_n + \g_n,
\end{align}
where $\p_n$ corresponds to $k=1$, $\square_n$ corresponds to the sum over all $k$ even, and $\g_n$ corresponds to the sum over all odd $k\geq 3$.

In the next three sections, we continue to use Rudnick's methods in order to compute $\p_n$, $\square_n$, and $\g_n$. Not surprinsingly, our results will only slightly differ from Rudnick's. The addition of $-1/q^\frac{n}{2}$ from (1.1) will be the main difference. We will also have different cut-off points for $n$ when estimating $\p_n$ (see section 5.3 of \cite{Rud10}).

\section{Contribution of the Primes: $\p_n$}
\label{section5}

The contribution of the primes in (\ref{trace_2g+2}) is given by:

\begin{align*}
\p_n=-\frac{1}{q^{\frac{n}{2}}}\sum_{\deg(P)=n}n\chi_Q(P).
\end{align*}
So,
\begin{align*}
\langle    \p_n \rangle_{\f_{2g+2}}&=\frac{-n}{(q-1)q^{2g+1+\frac{n}{2}}}\sum_{\deg(P)=n}\sum_{2\alpha+\beta=2g+2}\sigma_n(\alpha)\sum_{\deg(B)=\beta}\Biggr(\frac{B}{P}\Biggr)\\
&=\frac{-n}{(q-1)q^{2g+1+\frac{n}{2}}}\sum_{2\alpha+\beta=2g+2}\sigma_n(\alpha)S(\beta;n).
\end{align*}\\

From Lemma \ref{lem_alpha}, if $n>g+1$, then
\begin{align*}
\sigma_n(\alpha)\neq 0 &\Rightarrow \mbox{($\alpha \equiv 0$ $(n)$ or $\alpha \equiv 1$ $(n)$)}\\
& \Rightarrow \mbox{($\alpha=0$ or $\alpha=1$)},
\end{align*} 
which follows from the fact that $0\leq \alpha\leq g+1$.\\
\noindent Since $\sigma_n(0)=1$ and $\sigma_n(1)=-q$, when $n>g+1$, we have
\begin{align*}
\langle \p_n \rangle_{\f_{2g+2}}=\frac{-n}{(q-1)q^{2g+1 +\frac{n}{2}}} (S(2g+2;n)-qS(2g;n)).
\end{align*}

We now compute $\langle \p_n \rangle_{\f_{2g+2}}$ by considering the case $n\leq g+1$ and the case $n>g+1$, which we break into four (non-distinct) ranges:\\

(i) $n\leq g+1:$ If $S(\beta;n)\neq 0$, then $\beta<n$; since $\beta$ is even,
\begin{align*}
S(\beta;n)&=\pi_q(n)q^{\frac{\beta}{2}}+O(\frac{\beta}{n}q^{\beta+\frac{n}{2}})\\
&=\frac{q^{n+\frac{\beta}{2}}}{n}+O(\frac {q^{\frac{n}{2}+\frac{\beta}{2}}}{n})+O(\frac{\beta}{n}q^{\beta+\frac{n}{2}}).
\end{align*}
Then
\begin{align*}S(\beta;n)\ll \frac{\beta}{n}q^{n+\beta},\end{align*}
which implies that
\begin{align*}
\langle   \p_n\rangle_{\f_{2g+2}} &\ll\frac{n}{q^{2g+\frac{n}{2}}}\sum_{\beta<n}\frac{\beta}{n}q^{\beta+n}\\
&\ll\frac{n}{q^{2g+\frac{n}{2}}}q^{2n}=nq^{\frac{3n}{2}-2g}\ll gq^{\frac{-g}{2}}.
\end{align*}

(ii) $g+1<n<2g+1:$ Since $2g+2,2g\geq n$, $S(2g+2;n)=S(2g;n)=0.$ Hence,
\begin{align*}\langle    \p_n \rangle_{\f_{2g+2}}= \frac{-n}{(q-1)q^{2g+1+\frac{n}{2}}}(S(2g+2;n)-qS(2g;n))=0.\end{align*}

(iii) $n=2g+1:$ Since $2g+2\geq n$, $S(2g+2;n)=0$ and we get that
\begin{align*}
\langle    \p_n \rangle_{\f_{2g+2}}&= \frac{n}{(q-1)q^{2g+1+\frac{n}{2}}}\cdot q \cdot S(2g;n)\\
&=\frac{2g+1}{(q-1)q^{3g+\frac{1}{2}}}\cdot S(2g;2g+1).
\end{align*}
Using (\ref{S(n-1;n)_odd}),
\begin{align*}
\langle    \p_n \rangle_{\f_{2g+2}}
&=\frac{2g+1}{(q-1)q^{3g+\frac{1}{2}}}\cdot (\pi_q(2g+1)q^{\frac{(2g+1)-1}{2}}).
\end{align*}
By replacing $\pi_q(2g+1)$ and simpligying, we obtain
\begin{align*}
\langle    \p_n \rangle_{\f_{2g+2}}&= \frac{2g+1}{(q-1)q^{3g+\frac{1}{2}}}\cdot \Biggr(\Biggr(\frac{q^{2g+1}}{2g+1}+O\Big(\frac{q^g}{2g+1}\Big)\Biggr)q^g\Biggr)\\
&= \frac{q^{\frac{1}{2}}}{q-1}+O(q^{-g}).
\end{align*}

(iv) $n=2g+2$: Similarly,
\begin{align*}
\langle    \p_n \rangle_{\f_{2g+2}}&= \frac{n}{(q-1)q^{2g+1+\frac{n}{2}}}\cdot q \cdot S(2g;n)\\
&=\frac{2g+2}{(q-1)q^{3g+1}}\cdot S(2g;2g+2).
\end{align*}
From Theorem \ref{beta_even},
\begin{align*}
\langle    \p_n \rangle_{\f_{2g+2}}&=\frac{2g+2}{(q-1)q^{3g+1}}\cdot \Biggr(\Big(\frac{q^{2g+2}}{2g+2}+O(\frac{q^g}{2g+2})\Big)(q^g-q^{g-1})+O(q^{2g})\Biggr)\\
&=1+O(q^{-g})+O(gq^{-g}).  
\end{align*}

(v) $n>2g+2:$ We apply Theorem \ref{beta_even} to get 
\begin{align*}
\langle    \p_n \rangle_{\f_{2g+2}} &= \frac{-n}{(q-1)q^{2g+1+\frac{n}{2}}}\big(S(2g+2;n)-qS(2g;n)\big)\\
&=\frac{-n}{(q-1)q^{2g+1+\frac{n}{2}}}\Biggr(\pi_q(n)(q^{\frac{2g+2}{2}}-\eta_nq^{2g+2-\frac{n}{2}})-q\cdot\pi_q(n)(q^{\frac{2g}{2}}-\eta_nq^{2g-\frac{n}{2}})+O(q^n)\Biggr).
\end{align*}
Upon further simplification,
\begin{align*}
\langle    \p_n \rangle_{\f_{2g+2}}&=\frac{n}{(q-1)q^{2g+1+\frac{n}{2}}}\Biggr(\eta_n \pi_q(n) (q^{2g+2-\frac{n}{2}}-q^{2g+1-\frac{n}{2}})+O(q^n)\Biggr)\\
&=\frac{n\eta_n \pi_q(n)}{q^n}+O(nq^{\frac{n}{2}-2g})\\
&=\eta_n (1+O(q^{\frac{-n}{2}}))+O(nq^{\frac{n}{2}-2g}).
\end{align*}

\begin{note}
When $n=2g+2$, (v) yields (iv). 
\end{note}

\section{Contribution of the Squares: $\square_n$}

For $n$ even, we have the following contribution from the squares of prime powers:
\begin{align*}
\square_n &= -\frac{1}{q^{\frac{n}{2}}}\sum_{\deg(P^{2k})=n}\Lambda(P^{2k})\chi_Q(P^{2k})\\
&=-\frac{1}{q^{\frac{n}{2}}}\sum_{\deg(P^{k})=\frac{n}{2}}\Lambda(P^{k})\chi_Q((P^{k})^2)\\
&=-\frac{1}{q^{\frac{n}{2}}}\sum_{\deg(h)=\frac{n}{2}}\Lambda(h)\chi_Q(h^2).
\end{align*}

Therefore,
\begin{align*}
\langle    \square_n \rangle_{\f_{2g+2}}&=-\frac{1}{q^{\frac{n}{2}}}\sum_{\deg(P^{k})=\frac{n}{2}}\Lambda(P^{k})\langle   \chi_Q(P^{2k})\rangle_{\f_{2g+2}}\\
&=-\frac{1}{q^{\frac{n}{2}}}\sum_{\deg(P^{k})=\frac{n}{2}}\deg(P)\frac{1}{(q-1)q^{2g+1}}\sum_{0\leq \alpha \leq g+1}\sum_{\substack{{\deg(A)=\alpha}\\{P\nmid A}}}\mu(A)\sum_{\substack{{\deg(B)=2g+2-2\alpha}\\{P\nmid B}}}1.
\end{align*}

Although section 5 shows a slight deviation in $\langle\p_n \rangle_{\f_{2g+2}}$ from $ \langle \p_n \rangle_{\f_{2g+1}}$, we will see that $\langle\square_n \rangle_{\f_{2g+2}}=\langle    \square_n \rangle_{\f_{2g+1}}.$\\

Let $m=\deg(P)$. Since
\begin{align*}
\#\{\mbox{$B:\deg(B)=\beta$, $P\nmid B$}\}= q^{\beta}\cdot\left\{\def\arraystretch{1.2}%
\begin{array}{@{}c@{\quad}l@{}}
1 & \mbox{if $m>\beta$} \\
1-\frac{1}{|P|} & \mbox{if $m\leq \beta$,}\\
\end{array}\right.
\end{align*}

\begin{align*}
\langle \p_n \rangle_{\f_{2g+1}}
&=-\frac{1}{q^{\frac{n}{2}}}\sum_{\deg(P^{k})=\frac{n}{2}}\deg(P)\frac{1}{(q-1)q^{2g+1}}\sum_{0\leq \alpha \leq g+1}\sum_{\substack{{\deg(A)=\alpha}\\{P\nmid A}}}\mu(A)\sum_{\substack{{\deg(B)=2g+2-2\alpha}\\{P\nmid B}}}1\\
&=-\frac{q}{(q-1)q^{\frac{n}{2}}}\sum_{\deg(P^{k})=\frac{n}{2}}\deg(P)\cdot \Biggr(
\sum_{g+1-\frac{m}{2}<\alpha \leq g+1}\frac{\sigma_m(\alpha)}{q^{2\alpha}}\\
&\qquad\qquad\qquad\qquad\qquad\qquad+\Big(1-\frac{1}{|P|}\Big)\cdot \sum_{0\leq \alpha \leq g+1-\frac{m}{2}}\frac{\sigma_m(\alpha)}{q^{2\alpha}}\Biggr)\\
&=-\frac{q}{(q-1)q^{\frac{n}{2}}}\sum_{\deg(P^{k})=\frac{n}{2}}\deg(P)\cdot \Biggr(
\Big(1-\frac{1}{|P|}\Big)\cdot \sum_{\alpha \geq 0}\frac{\sigma_m(\alpha)}{q^{2\alpha}}\\
&\qquad\qquad\qquad\qquad\qquad\qquad+\sum_{g+1-\frac{m}{2}<\alpha \leq g+1}\frac{\sigma_m(\alpha)}{q^{2\alpha}}
-\Big(1-\frac{1}{|P|}\Big)\cdot \sum_{\alpha >g+1-\frac{m}{2} }\frac{\sigma_m(\alpha)}{q^{2\alpha}}
\Biggr)\\
&=-\frac{q}{(q-1)q^{\frac{n}{2}}}\sum_{\deg(P^{k})=\frac{n}{2}}\deg(P)\cdot \Biggr(
\Big(1-\frac{1}{|P|}\Big)\cdot \sum_{\alpha \geq 0}\frac{\sigma_m(\alpha)}{q^{2\alpha}}\\
&\qquad\qquad\qquad\qquad\qquad\qquad-\sum_{\alpha > g+1}\frac{\sigma_m(\alpha)}{q^{2\alpha}}
+\frac{1}{|P|}\cdot \sum_{\alpha >g+1-\frac{m}{2} }\frac{\sigma_m(\alpha)}{q^{2\alpha}}
\Biggr)\\
&=-\frac{q}{(q-1)q^{\frac{n}{2}}}\sum_{\deg(P^{k})=\frac{n}{2}}\deg(P)\cdot \Biggr(
\Big(1-\frac{1}{|P|}\Big)\cdot \sum_{\alpha \geq 0}\frac{\sigma_m(\alpha)}{q^{2\alpha}}+O(q^{-2g})\Biggr).
\end{align*}

Solving the recurrence relation in Lemma \ref{lem_alpha},
\begin{align*}
\langle    \square_n \rangle_{\f_{2g+2}}
&=-\frac{q}{(q-1)q^{\frac{n}{2}}}\sum_{\deg(P^{k})=\frac{n}{2}}\deg(P)\cdot \Biggr(\Big(1-\frac{1}{|P|}\Big)\cdot \frac{1-\frac{1}{q}}{1-\frac{1}{|P|^2}}+O(q^{-2g})\Biggr)\\
&=-\frac{1}{q^{\frac{n}{2}}}\sum_{\deg(P)|\frac{n}{2}} \deg(P)\cdot \Big(\frac{|P|}{|P|+1}+O(q^{-2g})\Big)\\
&=-\frac{1}{q^{\frac{n}{2}}}\sum_{\deg(P)|\frac{n}{2}} \deg(P)\cdot \Big(1-\frac{1}{|P|+1}+O(q^{-2g})\Big)\\
&=-\frac{1}{q^{\frac{n}{2}}}\sum_{\deg(P)|\frac{n}{2}} \deg(P)
+\frac{1}{q^{\frac{n}{2}}}\sum_{\deg(P)|\frac{n}{2}} \deg(P)\cdot \Big(\frac{1}{|P|+1}+O(q^{-2g})\Big).
\end{align*}

Since
\begin{align*}
q^l=\sum_{\deg(h)=l} \Lambda(h)=\sum_{\deg(P^k)=l}\deg(P)=\sum_{\deg(P)|l}\deg(P),
\end{align*}

\begin{align*}
\langle    \square_n \rangle_{\f_{2g+2}} &=
-1+\frac{1}{q^{\frac{n}{2}}}\sum_{\deg(P)|\frac{n}{2}} \frac{\deg(P)}{|P|+1}+O(q^{-2g}).
\end{align*}

\section{Contribution of the Higher Prime Powers: $\g_n$}
\label{section7}

Using trivial bounds for $\g_n$, we obtain slightly different results from Rudnick in the case where $n>6g$. This is due to the fact that our estimates of $\g_n$ involve $S(\beta;n)$ for $\beta$ even, as opposed to $\beta$ odd, and, from Lemma \ref{lem_S_beta}, 
\begin{align*}
S(\beta;n)\ll\left\{\def\arraystretch{1.2}%
\begin{array}{@{}l@{\quad}l@{}}
q^{n+\beta} & \text{if $\beta$ is even} \\
q^{\frac{n}{2}+\beta} & \text{if $\beta$ is odd.}
\end{array}\right.
\end{align*}
We shall see that our bounds for $n>6g$ are absorbed in the error term from $\langle \p_n \rangle_{\f_{2g+2}}$ when $n>2g+1$.\\

The contribution to $\tr(\Theta_{C_Q}^n)$ from the higher odd prime powers in (\ref{trace_2g+2}) is:  

\begin{align*}
\g_n &= -\frac{1}{q^{\frac{n}{2}}}\sum_{\substack{{d|n}\\{3 \text{$\leq d:$ odd}}}}\sum_{\deg(P)=\frac{n}{d}}\frac{n}{d}\chi_Q(P^d)\\
&= -\frac{1}{q^{\frac{n}{2}}}\sum_{\substack{{d|n}\\{3 \text{$\leq d:$ odd}}}}\sum_{\deg(P)=\frac{n}{d}}\frac{n}{d}\chi_Q(P),
\end{align*}
where the last equality follows from the fact that $\chi_Q(P^d)=\chi_Q(P)$ for odd $d$.

This implies that
\begin{align*}
\langle \g_n \rangle_{\f_{2g+2}} &= 
-\frac{1}{(q-1)q^{2g+1+\frac{n}{2}}}\sum_{\substack{{d|n}\\{3 \text{$\leq d:$ odd}}}}\frac{n}{d}
\sum_{\deg(P)=\frac{n}{d}}
\sum_{2\alpha + \beta = 2g+2} \sigma_{\frac{n}{d}}(\alpha)
\sum_{\deg(B)=\beta}\Big(\frac{B}{P}\Big)\\
&=-\frac{1}{(q-1)q^{2g+1+\frac{n}{2}}}\sum_{\substack{{d|n}\\{3 \text{$\leq d:$ odd}}}}\frac{n}{d}\sum_{2\alpha + \beta = 2g+2} \sigma_{\frac{n}{d}}(\alpha)S(\beta;\frac{n}{d}).
\end{align*}

If $S(\beta;\frac{n}{d})\neq 0$, then $\beta < \frac{n}{d}$; also, $S(\beta; \frac{n}{d})\ll q^{\frac{n}{d}+\beta}$. Hence,

\begin{align*}
\langle \g_n \rangle_{\f_{2g+2}} &\ll \frac{1}{q^{2g+\frac{n}{2}}}
\sum_{\substack{{d|n}\\{3 \text{$\leq d:$ odd}}}}\frac{n}{d}
\sum_{\beta \leq \min(\frac{n}{d},2g+2)}q^{\frac{n}{d}+\beta}\\
&\ll \frac{n}{q^{2g+\frac{n}{2}}}\sum_{\substack{{d|n}\\{3 \text{$\leq d:$ odd}}}}q^{\frac{n}{d}+\min(2g,\frac{n}{d})}.
\end{align*}

If $\frac{n}{3}\leq 2g$, then $\min(\frac{n}{d},2g)=\frac{n}{d}$ for all $d\geq 3$; and so,
\begin{align*}
\langle \g_n \rangle_{\f_{2g+2}} &\ll \frac{n}{q^{2g+\frac{n}{2}}}\sum_{\substack{{d|n}\\{3 \text{$\leq d:$ odd}}}}q^{\frac{2n}{d}}\ll  \frac{n}{q^{2g+\frac{n}{2}}}\cdot q^\frac{2n}{3} = \frac{n}{q^{2g}}\cdot q^{\frac{n}{6}}\\
&\ll \frac{g}{q^{2g}}\cdot q^g=gq^{-g}.
\end{align*}

If $\frac{n}{3}>2g$, then $\min(\frac{n}{d},2g)\leq \frac{n}{3}$ for all $d\geq 3$; therefore,
\begin{align*}
\langle \g_n \rangle_{\f_{2g+2}} &\ll \frac{n}{q^{2g+\frac{n}{2}}}\cdot q^\frac{2n}{3} \ll nq^{\frac{n}{6}-2g}.
\end{align*}

\section{Computing $\langle \tr(\Theta_{C_Q}^n)\rangle_{\f_{2g+2}}$}
\label{section8}

Since $\langle \tr(\Theta_{C_Q}^n)\rangle_{\f_{2g+2}} = -\frac{1}{q^{\frac{n}{2}}}+\langle \p_n \rangle_{\f_{2g+2}} + \langle \square_n \rangle_{\f_{2g+2}} + \langle \g_n \rangle_{\f_{2g+2}}$, we obtain the following:

\begin{align*}
\langle \tr\Theta_{C_Q}^n\rangle_{\f_{2g+2}} = -\frac{1}{q^{\frac{n}{2}}}&+
\left\{\def\arraystretch{1.2}%
\begin{array}{@{}l@{\quad}l@{}}
O(gq^{\frac{-g}{2}}) & 0 <n\leq g+1 \\
0 & g+1<n<2g+1\\
\frac{q^{\frac{1}{2}}}{q-1}+O(q^{-g}) & n=2g+1\\
\eta_n(1+O(q^{\frac{-n}{2}}))+O(nq^{\frac{n}{2}-2g}) & 2g+1<n\\
\end{array}\right.\\
&+\eta_n(-1+\frac{1}{q^{\frac{n}{2}}}\sum_{\deg(P)|\frac{n}{2}} \frac{\deg(P)}{|P|+1}+O(q^{-2g}))\\
&+
\left\{\def\arraystretch{1.2}%
\begin{array}{@{}c@{\quad}l@{}}
O(gq^{-g}) & n\leq 6g\\
O(nq^{\frac{n}{6}-2g}) & 6g<n.\\
\end{array}\right.
\end{align*}

\noindent In particular,
\begin{theorem}
\label{pre_thm}
\begin{align*}
\langle \tr\Theta_{C_Q}^n\rangle_{\f_{2g+2}} = -\frac{1}{q^{\frac{n}{2}}}&+\eta_n \frac{1}{q^{\frac{n}{2}}}\cdot\sum_{\deg(P)|\frac{n}{2}} \frac{\deg(P)}{|P|+1} +O(gq^{\frac{-g}{2}})\\
&+ 
\left\{\def\arraystretch{1.2}%
\begin{array}{@{}l@{\quad}l@{}}
-\eta_n & 0<n< 2g+1\\
\frac{q^{\frac{1}{2}}}{q-1} & n=2g+1\\
O(nq^{\frac{n}{2}-2g}) & 2g+1<n.\\
\end{array}\right.
\end{align*}
\end{theorem} 

\section{Computing $\langle \tr(\Theta_{C_Q}^n)\rangle_{\f_{2g+1}\bigcup \f_{2g+2}}$}

From \cite{Rud10}, we have the following:

\begin{theorem}
\label{Thm_Rud}
\begin{align*}
\langle \tr\Theta_{C_Q}^n\rangle_{\f_{2g+1}} = \eta_n \frac{1}{q^{\frac{n}{2}}}\cdot\sum_{\deg(P)|\frac{n}{2}} \frac{\deg(P)}{|P|+1} +O(gq^{-g})
+ 
\left\{\def\arraystretch{1.2}%
\begin{array}{@{}l@{\quad}l@{}}
-\eta_n & 0<n<2g\\
-1-\frac{1}{q-1} & n=2g\\
O(nq^{\frac{n}{2}-2g}) & 2g<n.\\
\end{array}\right.
\end{align*}
\end{theorem}

We would like to find the expected value of $\tr(\Theta_{C_Q}^n)$ over all curves $C_Q:y^2=Q(x)$ of genus $g$ with $Q$ monic. To do this, we use Theorems \ref{pre_thm} and \ref{Thm_Rud}. By identifying the family of curves described above with $\f_{2g+1}\bigcup \f_{2g+2}$, we see that
\begin{align*}
\langle \tr(&\Theta_{C_Q}^n) \rangle _{\f_{2g+1}\bigcup \f_{2g+2}}\\
&=\frac{\#\f_{2g+1}}{\#(\f_{2g+1}\bigcup \f_{2g+2})}\langle \tr(\Theta_{C_Q}^n) \rangle _{\f_{2g+1}}+\frac{\#\f_{2g+2}}{\#(\f_{2g+1}\bigcup \f_{2g+2})}\langle \tr(\Theta_{C_Q}^n) \rangle _{\f_{2g+2}},
\end{align*}
where 
\begin{align*}
\#(\f_{2g+1}\bigcup \f_{2g+2})&=\#\f_{2g+1}+\#\f_{2g+2}\\
&=(q-1)q^{2g}+(q-1)q^{2g+1}\\
&=q^{2g}(q-1)(q+1).
\end{align*}

We obtain the following:
\begin{corollary}
\label{pre_cor}
\begin{align*}
\langle \tr(\Theta_{C_Q}^n) \rangle _{\f_{2g+1}\bigcup \f_{2g+2}}=-\frac{1}{q^{\frac{n}{2}}}\frac{q}{q+1}
&+\eta_n \frac{1}{q^{\frac{n}{2}}}\cdot\sum_{\deg(P)|\frac{n}{2}} \frac{\deg(P)}{|P|+1}
+O(gq^{\frac{-g}{2}})\\
&+ 
\left\{\def\arraystretch{1.2}%
\begin{array}{@{}l@{\quad}l@{}}
-\eta_n & 0<n< 2g\\
-1-\frac{1}{q^2-1} & n=2g\\
\frac{q^{\frac{3}{2}}}{q^2-1} & n=2g+1\\
O(nq^{\frac{n}{2}-2g}) & 2g+1<n.\\
\end{array}\right.
\end{align*}
\end{corollary}

\begin{note}
The first main term in Corollary \ref{pre_cor} does not appear in Theorem \ref{Thm_Rud}, neither does the term $\frac{q^{\frac{3}{2}}}{q^2-1}$ corresponding to $n=2g+1$. Similarly, for $n=2g$, the constant $\frac{1}{q-1}$ in Theorem \ref{Thm_Rud} is scaled down to $\frac{1}{q^2-1}$ in Corollary \ref{pre_cor}. In the next section, we shall see that these differences are diminished when we consider the average of $\tr(\Theta_{C_Q}^n)$ over $\h_g$.
\end{note}
     
\section{Computing $\langle \tr(\Theta_{C_Q}^n)\rangle_{\h_g}$}
\label{section10}

As we mentioned in the introduction, averaging over monic squarefree polynomials of a fixed degree is not the same as averaging over the moduli space of hyperelliptic curves of genus $g$: in the latter case, we consider polynomials of degree $2g+1$ and $2g+2$. Also, by restricting ourselves to monic polynomials, we introduce a bias in the average value of the trace: the contribution of the point at infinity is related to the leading coefficient of $Q$, as seen by equation (\ref{trace}).\\

We now turn our attention to finding the average of $\tr(\Theta_{C_Q}^n)$ over $\h_g$: from equations (\ref{first}) and (\ref{trace}),
\begin{align}
   \langle \tr(\Theta_{C_Q}^n) \rangle _{\h_g}&=\frac{1}{\#\h_g}\cdot \sum_{Q\in \h_g}\Big( -\frac{\lambda_Q^n}{q^\frac{n}{2}}-\frac{1}{q^\frac{n}{2}}\sum_{\deg(f)=n}\Lambda(f) \chi_Q(f)\Big),
   \end{align}
   where
\begin{align*}
\lambda_Q=\left\{\def\arraystretch{1.2}%
\begin{array}{@{}r@{\quad}l@{}}
-1 & \mbox{if $\deg(Q)$ is even and $\sgn(Q)\neq \square$}\\
0 & \mbox{if $\deg(Q)$ is odd}\\
1 & \mbox{if $\deg(Q)$ is even and $\sgn(Q)=\square$.}\\
\end{array}\right.
\end{align*}

   Since there are exactly $(q-1)/2$ squares and $(q-1)/2$ non-squares in $\F_q^*$, if $n$ is odd,
\begin{align*}\frac{1}{\#\h_g}\cdot\sum_{Q\in \h_g}\lambda_Q^n=\frac{1}{\#\h_g}\cdot\sum_{Q\in \h_g}\lambda_Q=0.\end{align*}
   On the other hand, if $n$ is even,
   \begin{align*}\frac{1}{\#\h_g}\cdot\sum_{Q\in \h_g}\lambda_Q^n=\frac{1}{\#\h_g}\cdot\sum_{Q\in \h_g}|\lambda_Q|=\frac{\#\widehat{\f}_{2g+2}}{\#\h_g}=\frac{q}{q+1}.\end{align*}
   
   Also, given $D\in \F_q[x]$ with $\deg(D)=d$, we may write $D=A^2 B$, where $A,B\in \F_q[x]$ with $A$ monic, $B$ not necessarily monic, and \text{$\deg(A)=\alpha$, $\deg(B)=\beta$}, so that $d=2\alpha + \beta$. From here, we can take the character sum above over all elements in $\F_q[x]$ of genus $g$ by sieving out the squarefree terms (as we did earlier):
   \begin{align*}
  \sum_{Q\in \h_g}\chi_Q(f)&=\sum_{\substack{{2\alpha+\beta=d}\\{d=2g+1,2g+2}}} \sum_{\substack{{\deg(B)=\beta}\\ \text{$B$ n.n.m.}}} \sum_{\deg(A)=\alpha} \mu(A)\Big(\frac{A}{f}\Big)^2\Big(\frac{B}{f}\Big)\\
  &=\sum_{\substack{{2\alpha+\beta=d}\\{d=2g+1,2g+2}}}\sigma(f;\alpha) \sum_{\substack{{\deg(B)=\beta}\\{\text{$B$ n.n.m.}}}}\Big(\frac{B}{f}\Big)\\
  &=\sum_{\substack{{2\alpha+\beta=d}\\{d=2g+1,2g+2}}} \sigma(f;\alpha)\sum_{a\in \F_q^*}\sum_{\deg(B)=\beta}\Big(\frac{aB}{f}\Big)\\
  &=\sum_{\substack{{2\alpha+\beta=d}\\{d=2g+1,2g+2}}} \sigma(f;\alpha)\sum_{a\in \F_q^*}\sum_{\deg(B)=\beta}\Big(\frac{a}{f}\Big)\cdot \Big(\frac{B}{f}\Big)\\
  &=\sum_{a\in \F_q^*}\Big(\frac{a}{f}\Big) \sum_{\substack{{2\alpha+\beta=d}\\{d=2g+1,2g+2}}} \sigma(f;\alpha)\sum_{\deg(B)=\beta}\Big(\frac{B}{f}\Big).
\end{align*}  
  
Hence,
\begin{align*}
&\langle \tr(\Theta_{C_Q}^n) \rangle _{\h_g}=\\
  &-\frac{1}{q^\frac{n}{2}\cdot \#\h_g}\Biggr(
  \sum_{Q\in \h_g}\lambda_Q^n+ \sum_{\deg(f)=n} \Lambda(f)\sum_{a\in \F_q^*}\Big(\frac{a}{f}\Big) \sum_{\substack{{2\alpha+\beta=d}\\{d=2g+1,2g+2}}} \sigma(f;\alpha)\sum_{\deg(B)=\beta}\Big(\frac{B}{f}\Big)\Biggr).
\end{align*}

If $f$ is a power of some prime in $\F_q[x]$, say $f=P^k$, then for all $a\in \F_q^*$ (see Proposition 3.2 of \cite{Ros02}),
\begin{align*}\Big(\frac{a}{f}\Big)=\Big(\frac{a}{P}\Big)^k=\Big(a^{\frac{q-1}{2}\deg(P)}\Big)^k=a^{\frac{q-1}{2}\deg(f)}.\end{align*}
If $\deg(f)=\deg(P^k)=n$ is even, then $\Big(\frac{a}{f}\Big)=1$ because $|\F_q^*|=q-1$. This tells us that $\sum_{a\in \F_q^*}\Big(\frac{a}{f}\Big)=q-1$. If $\deg(f)=\deg(P^k)=n$ is odd, then we have that $\sum_{a\in \F_q^*}\Big(\frac{a}{f}\Big)=0$ because there are exactly $(q-1)/2$ QR in $\F_q^*$ and exactly $(q-1)/2$ NQR in $\F_q^*$.

So, for $n$ odd,
\begin{align*}\langle \tr(\Theta_{C_Q}^n) \rangle _{\h_g}=0,\end{align*}
and for $n$ even,
\begin{align*}
\langle &\tr(\Theta_{C_Q}^n) \rangle _{\h_g}\\
&=-\frac{1}{q^\frac{n}{2}}\frac{q}{q+1}-\frac{1}{\#\h_g\cdot q^{\frac{n}{2}}}\sum_{\deg(f)=n} \Lambda(f)\sum_{a\in \F_q^*}\Big(\frac{a}{f}\Big) \sum_{\substack{{2\alpha+\beta=d}\\{d=2g+1,2g+2}}} \sigma(f;\alpha)\sum_{\deg(B)=\beta}\Big(\frac{B}{f}\Big)\\
&=-\frac{1}{q^\frac{n}{2}}\frac{q}{q+1}-\frac{1}{\#(\f_{2g+1}\bigcup \f_{2g+2})\cdot q^{\frac{n}{2}}}\sum_{\deg(f)=n} \Lambda(f) \sum_{\substack{{2\alpha+\beta=d}\\{d=2g+1,2g+2}}} \sigma(f;\alpha)\sum_{\deg(B)=\beta}\Big(\frac{B}{f}\Big)\\
&=-\frac{1}{q^\frac{n}{2}}\frac{q}{q+1}-\frac{1}{q^\frac{n}{2}}\sum_{\deg(f)=n} \Lambda(f) \langle \chi_Q(f) \rangle_{\f_{2g+1}\bigcup \f_{2g+2}}.
\end{align*}

In other words,
\begin{theorem}
\label{main_thm}
For $n$ odd,
\begin{align*}\langle \tr(\Theta_{C_Q}^n) \rangle _{\h_g}=0,\end{align*}
and for $n$ even,
\begin{align*}
\langle \tr(\Theta_{C_Q}^n) \rangle _{\h_g}
&= \langle \tr(\Theta_{C_Q}^n) \rangle _{\f_{2g+1}\bigcup \f_{2g+2}}\\
&=\frac{1}{q^{\frac{n}{2}}}\cdot\sum_{\substack{{\deg(P)|\frac{n}{2}}\\{\deg(P)\neq 1}}} \frac{\deg(P)}{|P|+1} +O(gq^{\frac{-g}{2}})
+
\left\{\def\arraystretch{1.2}%
\begin{array}{@{}l@{\quad}l@{}}
-1 & 0<n<2g\\
-1-\frac{1}{q^2-1} & n=2g\\
O(nq^{\frac{n}{2}-2g}) & 2g<n.\\
\end{array}\right.
     \end{align*}
\end{theorem}

In particular,
\begin{corollary}
If $n$ is odd, then
\begin{align*}
\langle \tr(\Theta_{C_Q}^n) \rangle_{\h_g} = \int_{\USp(2g)} \tr(U^n) dU.
\end{align*}
For $n$ even with $3 \log_q(g) < n < 4g-5 \log_q(g)$ and $n \neq 2g$,
\begin{align*}
\langle \tr(\Theta_{C_Q}^n) \rangle_{\h_g} = \int_{\USp(2g)} \tr(U^n) dU + o(\frac{1}{g}).
\end{align*} 
\end{corollary}

\begin{proof}
The first part is clear. To prove the second part, we treat each non-mainterm in Theorem \ref{main_thm} separately and show that each of them contributes an error term of $o(\frac{1}{g})$ in the desired region.

Fix $\epsilon>0$. If $n<4g-(4+\epsilon)\log_q(g)$, then  
\begin{align*}
\lim_{g\rightarrow \infty}g\cdot nq^{\frac{n}{2}-2g} \leq \lim_{g\rightarrow\infty} g^2 g^{-2-\frac{\epsilon}{2}}
=\lim_{g\rightarrow \infty} g^\frac{-\epsilon}{2}=0;
\end{align*}
i.e., 
\begin{align*}
O(nq^{\frac{n}{2}-2g})=o(\frac{1}{g}).
\end{align*}

Note that 
\begin{align*}
\frac{1}{q^\frac{n}{2}}\sum_{\substack{{\deg(P)|\frac{n}{2}}\\{\deg(P)\neq 1}}}\frac{\deg(P)}{|P|+1}=O(\frac{n}{q^\frac{n}{2}}).
\end{align*}
If $n=(2+\epsilon)\log_q(g)$, then 
\begin{align*}
\lim_{g\rightarrow \infty} g \frac{n}{q^\frac{n}{2}} \ll_\epsilon \lim_{g\rightarrow \infty}\frac{g\log_q(g)}{g^{1+\frac{\epsilon}{2}}}=0.
\end{align*}
So, for $n>(2+\epsilon)\log_q(g)$,
\begin{align*}
\frac{1}{q^\frac{n}{2}}\sum_{\substack{{\deg(P)|\frac{n}{2}}\\{\deg(P)\neq 1}}}\frac{\deg(P)}{|P|+1}=O(\frac{n}{q^\frac{n}{2}}) =o(\frac{1}{g}).
\end{align*}
We have actually shown a stronger version of our statement; namely, for any fixed $\epsilon>0$ and for any even $n$ with $(2+\epsilon) \log_q(g) < n < 4g-(4+\epsilon) \log_q(g)$ and $n \neq 2g$,
\begin{align*}
\langle \tr(\Theta_{C_Q}^n) \rangle_{\h_g} = -1 + o(\frac{1}{g}).
\end{align*}
\end{proof}

We now look at another approach which quickly verifies the first result of Theorem \ref{main_thm}. This argument was provided by \text{Dr. Ze$\E$v} Rudnick: fix a finite field $\F_q$ of odd cardinality $q$, let $Q$ be any monic, squarefree polynomial in $\F_q[x]$ of degree $2g+1$ or $2g+2$, and let $a\in \F_q^*$. Then
\begin{align*}
\#C(\F_{q^n})&=\sum_{x_0\in \pp^1(\F_{q^n})} \Big(\chi_n(a(Q(x_0))+1\Big)\\
&=q^n+1+\sum_{x_0\in \pp^1(\F_{q^n})}\chi_n(a(Q(x_0)),
\end{align*}
where $\chi_n$ is a multiplicative character on $\F_{q^n}$ defined by
\begin{align*}
\chi_n(\alpha):=\left\{\def\arraystretch{1.2}%
\begin{array}{@{}r@{\quad}l@{}}
1 & \mbox{if $\alpha$ is a square in $\F_{q^n}^*$}\\
0 & \mbox{if $\alpha=0$}\\
-1 & \mbox{if $\alpha$ is not a square in $\F_{q^n}^*$.}\\
\end{array}\right.
\end{align*}
When $x_0$ is the point at infinity, $Q(x_0)$ is defined by the evaluation of $x^{2g+2}Q(\frac{1}{x})$ at $x=0$; i.e., $\chi_n(Q(\infty))$ yields $\lambda_Q$ according to the count of (\ref{infinity}). Moreover,
\begin{align*}
-q^{\frac{n}{2}}\tr(\Theta_{C}^n)=\sum_{x_0\in \pp^1(\F_{q^n})} \chi_n(aQ(x_0))
\end{align*}

Therefore,
\begin{align*}
\langle \tr(\Theta_{C}^n) \rangle_{\h_g}&=\frac{-1}{q^\frac{n}{2}\#\h_g}\sum_{a\in\F_q^*}\sum_{Q\in\F_q[x]}{'}\sum_{x_0\in \pp^1(\F_{q^n})} \chi_n(aQ(x_0))\\
&=\frac{-1}{q^\frac{n}{2}\#\h_g}\sum_{a\in\F_q^*}\chi_n(a)\sum_{Q\in\F_q[x]}{'}\sum_{x_0\in\pp^1(\F_{q^n})} \chi_n(Q(x_0)),
\end{align*}
where $\sum_{Q\in\F_q[x]}{'}$ indicates that the sum is over all monic, squarefree polynomials $Q\in \F_q[x]$ such that $\deg(Q)=2g+1$ or $\deg(Q)=2g+2$.
When $n$ is odd, there are exactly $\frac{q-1}{2}$ squares and $\frac{q-1}{2}$ non-squares in $\F_q^* \subset\F_{q^n}$, which tells us that 
\begin{align*}
\sum_{a\in\F_q^*}\chi_n(a)=0.
\end{align*}
 So, for odd $n$,
\begin{align*}
\langle \tr(\Theta_{C}^n) \rangle_{\h_g}=0.
\end{align*}
On the other hand, when $n$ is even, computing the average over the entire moduli space reduces to computing the average over the moduli space with the restriction that $a=1$: for even $n$, every element of $\F_q^*$ is a square in $\F_{q^n}^*$ so that 
\begin{align*}
\sum_{a\in\F_q}\chi_n(a)=q-1
\end{align*}
and
\begin{align*}
\langle \tr(\Theta_{C}^n) \rangle_{\h_g}&=-\frac{(q-1)}{q^\frac{n}{2}\#\h_g}\sum_{Q\in\F_q[x]}{'}\sum_{x_0\in\pp^1(\F_{q^n})} \chi_n(Q(x_0))\\
&=\langle \tr(\Theta_{C}^n) \rangle_{\widehat{\h_g}},
\end{align*}
where 
\begin{align*}
\widehat{\h_g}:=\{C_Q\in\h_g : \mbox{$Q$ monic}\}.
\end{align*}
Evidently, $\#\h_g=(q-1)\#\widehat{\h_g}$.

\section{Acknowledgements}

We thank \text{Dr. Chantal} David for many valuable suggestions which have vastly improved this paper and for her continuous support throughout this research. We also thank \text{Dr. Ze$\E$v} Rudnick and Manal Alzahrani for their valuable input and insight.

\end{document}